\documentclass[preprint,11pt,3p]{elsarticle}

\usepackage{graphicx}
\usepackage{amsmath}
\usepackage{amssymb}
\usepackage{euscript}
\usepackage{amsfonts}
\usepackage{amsthm}
\usepackage{hyperref}

\hypersetup{colorlinks,
	breaklinks,
	colorlinks=true,
	urlcolor=blue,
	citecolor=blue,
	linkcolor=blue
}

\usepackage{float,blkarray,tabularx,color,tikz,caption,subcaption}
\usepackage{color}
\usepackage[normalem]{ulem}
\graphicspath{{Images/}}

\newcommand{\ZG}{\mathbb{Z}_{\geq 0}}
\newcommand{\Z}{\mathbb{Z}}
\newcommand{\Q}{\mathbb{Q}}

\newcommand{\N}{\mathbb{N}}

\newcommand\aug{\fboxsep=-\fboxrule\!\!\!\fbox{\strut}\!\!\!}
\newcommand{\stargraphout}[2]{\begin{tikzpicture}
    \node[circle,fill=black] at (360:0mm) (center) {};
    \foreach \n in {1,...,#1}{
        \node at ({\n*360/#1}:#2cm) (n\n) {};
        \draw[->, ultra thick, red] (center)--(n\n);
        \node at (0,-#2*1.5) {$t_1=\{a^{#1}\}$}; 
    }
\end{tikzpicture}}

\newtheorem{lem}{Lemma}

\newtheorem{thm}{Theorem}
\newtheorem{cor}{Corollary}

\newtheorem{rem}{Remark}

\theoremstyle{definition}
\newtheorem{defn}{Definition}
\newtheorem{ex}{Example}

\definecolor{lavender}{rgb}{0.71, 0.49, 0.86}
\definecolor{lapislazuli}{rgb}{0.15, 0.38, 0.61}
\definecolor{emerald}{rgb}{0.31, 0.78, 0.47}
\definecolor{frenchrose}{rgb}{0.96, 0.29, 0.54}

\begin{document}

\begin{frontmatter}

\title{Analysis and Algorithmic Construction of Self-Assembled DNA Complexes}

\author[label1]{Cory Johnson}
\address[label1]{California State University, San Bernardino, Department of Mathematics, San Bernardino, CA}
\ead{corrine.johnson@csusb.edu}

\author[label2]{Andrew Lavengood-Ryan}
\address[label2]{Nevada State University, Department of Data, Media, and Design, Henderson, NV}
\ead{Andrew.Lavengood-Ryan@nevadastate.edu}

\begin{abstract}
DNA self-assembly is an important tool that has a wide range of applications such as building nanostructures, the transport of target virotherapies, and nano-circuitry. Tools from graph theory can be used to encode the biological process of DNA self-assembly. The principle component of this process is to examine collections of branched junction molecules, called pots, and study the types of structures that can be constructed. We restrict our attention to pots which contain one set of complementary cohesive-ends, i.e. a single bond-edge type, and we identify the types and sizes of structures that can be built from such a pot. In particular, we show a dependence between the order of graphs in the output of the pot and the number of arms on the corresponding tiles. Furthermore, we provide two algorithms which will construct complete complexes for a pot with a single bond-edge type.
\end{abstract}

\begin{keyword}
graph theory \sep graph algorithms \sep DNA self-assembly \sep flexible tile model 
\end{keyword}

\end{frontmatter}

\section{Introduction}

DNA self-assembly is a vital experimental process that is being utilized in labs across the country. The use of DNA self-assembly as a bottom-up technology for creating target nanostructures was  introduced in Seeman's laboratory in the 1980s \cite{Seeman82}.  The process relies on the complementary nature of nucleotides that comprise the structure of DNA. DNA self-assembly has applications ranging from the construction of nanostructures to experimental virotherapies \cite{ellis2014minimal,seeman2007overview}. Graphs are natural mathematical models for DNA self-assembled structures and we use a combination of graph theoretic and algebraic tools to optimize the assembly process.

The nature of the nucleotide base pairing allows DNA to be configured into a variety of shapes, such as hairpins, cubes, and other non-traditional structures \cite{chen1991, goodman2005, mao1999, shih2004,zhang1994}. The nature of the nucleotide base pairing may also be utilized to build larger structures \cite{chen1991,shih2004,zhang1994}.  Two models of the assembly process emerge: a model which utilizes rigid tiles \cite{winfree1998, winfree1999universal, rothemund2004algorithmic}, and the other using flexible tiles \cite{jonoska2003computation, jonoska2003self}. We study the flexible tile model which has been used to construct structures such as the cube and truncated octahedron \cite{chen1991,shih2004,zhang1994}. A detailed description of the graph theoretic model of flexible tile DNA self-assembly can be found in \cite{Johnson2021, Ellis2019, ellis2014minimal}.

In the DNA self-assembly process, target structures are built from \textit{branched junction molecules}, which are asterisk-shaped molecules whose arms consist of strands of DNA. The end of each arm contains  unsatisfied nucleotide bases creating a \textit{cohesive-end}. Each cohesive-end will bond with a complementary cohesive-end from another arm via Watson-Crick base pairing. We will formalize this process in Section \ref{sec:DNAcomplexes}.  Rather than referring to the precise nature of a cohesive-end (such as the exact nucleotide configuration), we use single alphabet letters to distinguish between cohesive-ends of different types. For example, $a$ and $b$ denote two non-compatible cohesive-ends, but cohesive-end $a$ will bond with cohesive-end $\hat{a}$ . We use the term \textit{bond-edge type} to refer to a pair of complementary cohesive-ends.

A collection of branched junction molecules  used in the self-assembly process is called a pot. Previous research has investigated questions arising from  determining the most efficient pot given a target complete complex \cite{Johnson2021, Ellis2019, ellis2014minimal}. We study the inverse question: given a pot, what are the complete complexes that can be assembled? In \cite{Johnson2021}, it was shown that determining if a given pot will realize a graph of a given order is NP-hard. Thus,  we restrict our attention to specialized cases; in particular, we study the case where the pot contains one bond-edge type. At this time we reserve our attention to three open questions:

\begin{enumerate}

\item What are the sizes of the DNA complexes that can be realized by a specific pot?

\item What types and what distributions of branched junction molecules does a pot use in realizing a target DNA complex?

\item Exactly what types of DNA complexes do we expect a pot to realize? (e.g. disconnected or connected complexes)

\end{enumerate}

 Section \ref{sec:DNAcomplexes} formalizes the graph theoretic model of the  DNA self-assembly process. Section \ref{sec:Results} is a collection of our results related to the three questions above, with Section \ref{sec:Algorithms} providing algorithms for producing connected graphs. We end in Section \ref{sec:Conclusion} with some insight into future directions.


\section{Encoding DNA Complexes using Graph Theory} \label{sec:DNAcomplexes}

The following graph theoretic model of DNA self-assembly is consistent with \cite{Johnson2021, Ellis2019, ellis2014minimal, JonoskaSpectrum}. Relevant definitions are copied here for the reader. A graph $G$ consists of a set $V = V(G)$ of vertices and a set $E = E(G)$ of 2-element subsets of $V$, called edges. Note that we allow for $G$ to be a multigraph.

A DNA complex is composed of $k$-armed branched junction molecules, which are asterisk-shaped molecules with $k$ arms of strands of DNA. Two arms can bond only if they have complementary base pairings. See Figure \ref{fig:tileex1} for an example of a branched junction molecule along with an example of Definition \ref{defn:main}. Definition \ref{defn:main} translates the biological process of self-assembly into a combinatorial representation.

\begin{defn} \label{defn:main}
Consider a $k$-armed branched junction molecule.

\begin{enumerate}

\item A $k$-armed branched junction molecule is modeled by a \textit{tile}. A tile is a vertex with $k$ half-edges representing the \textit{cohesive-ends} (or arms) of the molecule $a,b,c,\ldots.$ We will denote complementary cohesive-ends with $\hat{a},\hat{b},\hat{c},\ldots.$

\item A \textit{bond-edge type} is a classification of the cohesive-ends of tiles (without regard to hatted and unhatted letters). For example, $a$ and $\hat{a}$ will bond to form bond-edge type $a$.

\item We denote tile types by $t_j$, where $t_j = \{a_1^{e_1},\hat{a}_1^{e_2},\ldots,a_k^{e_{2k-1}},\hat{a}_k^{e_{2k}},\ldots\}$. The exponent on $a_i$ indicates the quantity of cohesive-ends of type $a_i$ present on the tile.

\item A \textit{pot} is a collection of distinct tile types such that for any cohesive-end type that appears on any tile in the pot, its complement also appears on some tile in the pot. We denote a pot by $P$.

\item It is our convention to think of bonded arms (that is, where cohesive-end $a_i$ has been matched with cohesive-end $\hat{a_i}$) as edges on a graph, and we think of the bond-edge type as providing direction and compatibility of edges. Unhatted cohesive-ends will denote half-edges directed away from the vertex, and hatted cohesive-ends will denote a half-edge directed toward the vertex. When cohesive-ends are matched, this will result in a directed edge pointing away from the tile that had an unhatted cohesive end and toward the vertex that had a hatted cohesive end.

\end{enumerate}

\end{defn}

\begin{defn}\cite{Ellis2019}
    An \textit{assembling pot} $P_\lambda(G)$ for a graph $G$ with assembly design $\lambda$ is the multiset of tiles $t_v$ where $t_v$ is associated to vertex $v \in V(G)$. Note that it is possible that tile $t_u = t_v$ even if $u \neq v$. If we view a vertex $v$ as its set of half-edges and a tile as a multiset of labels, then the labeling $\lambda$ can be used to map vertices to tiles by $\lambda: V \to P_\lambda(G)$ such that $\lambda(v) = t_v$.
\end{defn}

\begin{figure}[H]
\begin{subfigure}{0.4\textwidth}
\begin{center}
\includegraphics[width=\textwidth]{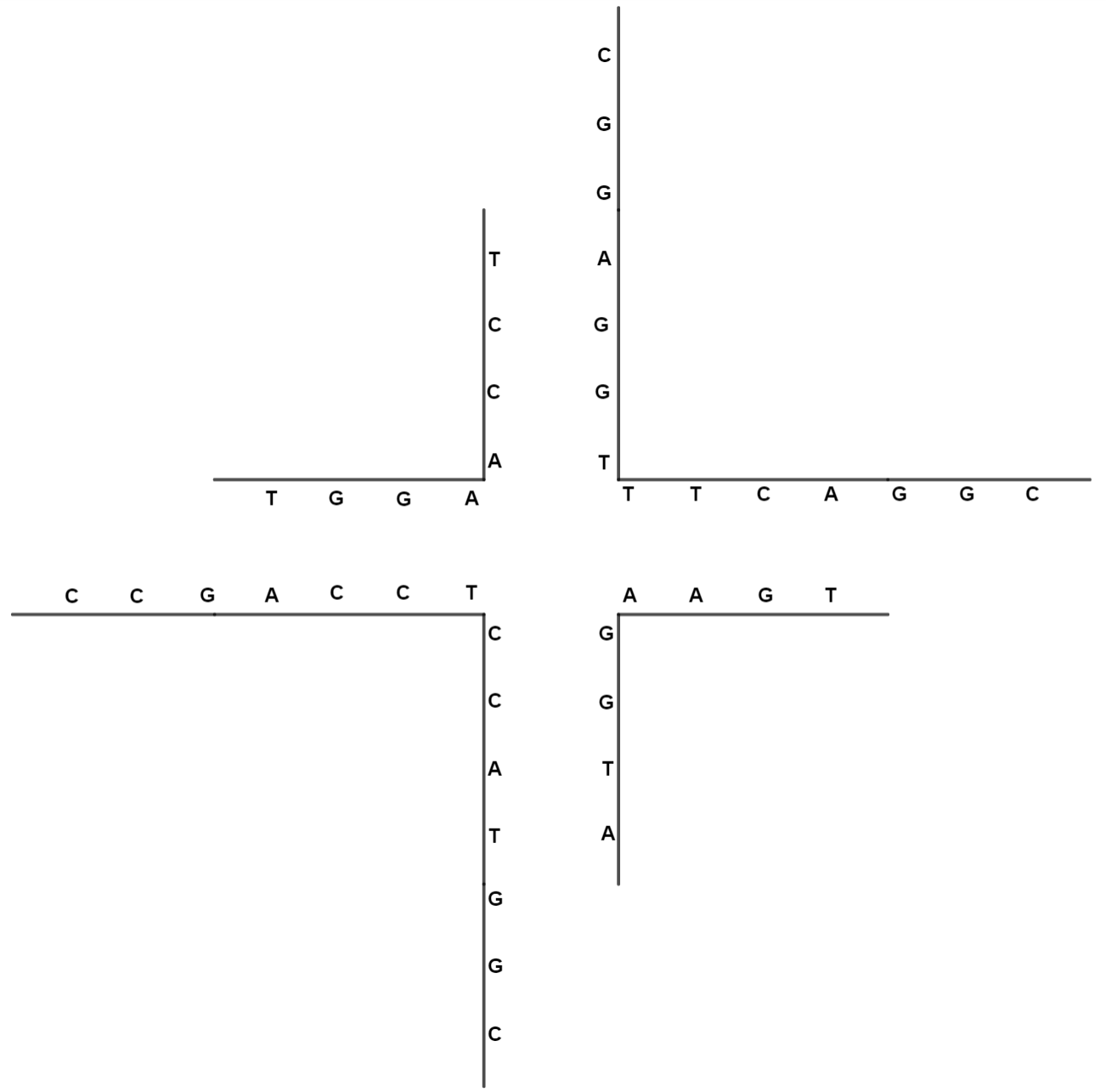}
\end{center}
\caption{Example of a 4-armed branched junction molecule}
\label{fig:tileex1a}
\end{subfigure}
\hfill
\begin{subfigure}{0.4\textwidth}
\begin{center}
\stargraphout{4}{2}
\end{center}
\caption{Tile with four cohesive-ends of type $a$}
\label{fig:tileex1b}
\end{subfigure}
\caption{A branched junction molecule and its associated tile representation.}
\label{fig:tileex1}
\end{figure}



\begin{defn}

We say that a graph $G$ is \textit{realized} by a pot $P$ if there exists a map $f: \{v_*\} \to P$ from the set of vertices with half-edges to the tile types with the following properties:

\begin{enumerate}
    \item If $v_* \mapsto t$, then there is an associated one-to-one correspondence between the cohesive ends of $t$ and the half-edges of $v$.
     
    \item If $\{u,v\} \in E(G)$, then the two half-edges of $\{u,v\}$ are assigned complementary cohesive ends.
\end{enumerate}

\end{defn}


The following result from \cite{ellis2014minimal} provides a foundation for the work presented in Section \ref{sec:Results}. Let $P = \{t_1,\ldots,t_p\}$ be a pot with $p$ tile types, and define $A_{i,j}$ to be the number of cohesive ends of type $a_i$ on tile $t_j$ and $\hat{A}_{i,j}$ to be the number of cohesive ends of type $\hat{a}_i$ on tile $t_j$. Suppose a target graph $G$ of order $n$  is realized by $P$ using $R_j$ tiles of type $j$. Since we consider only complete complexes, we have the following equations:
\begin{equation}\label{equ:Rj}
\sum_j R_j =n \text{ and } \sum_j R_j(A_{i,j}-\hat{A}_{i,j})=0 \qquad \text{for all }i.
\end{equation}

Define $z_{i,j} = A_{i,j} - \hat{A}_{i,j}$ and $r_j$ to be the proportion of tile-type $t_j$ used in the construction of $G$. Then the equations in Equation \ref{equ:Rj} become
\begin{equation}\label{equ:sum}
\sum_j r_j = 1 \text{ and } \sum_j r_jz_{i,j}=0 \qquad \text{for all } i.
\end{equation}
The equations in Equation \ref{equ:sum} naturally define a matrix associated to $P$.

\begin{defn}
Let $P = \{t_1, t_2, \ldots, t_p\}$ be a pot. Then the \textit{construction matrix} of $P$ is given by
\[
M_P = \begin{blockarray}{cccccc}
& t_1 & t_2 & \cdots & t_p & \hspace{0.5cm}\\
\begin{block}{c[cccc|c]}
						a_1 &	z_{1,1} & z_{1,2} & \cdots & z_{1,p} & 0 \\
						a_2 &	z_{2,1} & z_{2,2} & \cdots & z_{2,p} & 0 \\
						\vdots & 	\vdots & \vdots & \ddots & \vdots & \vdots \\
						a_m &	z_{m,1} & z_{m,2} & \cdots & z_{m,p} & 0 \\
						 & 		1 & 1 & \cdots & 1 & 1 \\
\end{block}
\end{blockarray}.
\]
\end{defn}

In general, there are infinitely many solutions to the system of equations defined by $M_P$, so it is desirable to concisely express these solutions. However, we only consider those solutions in $\left(\mathbb{Q} \cap [0,1]\right)^p$. That is, vectors whose entries are rational numbers between 0 and 1.

\begin{defn}\label{defn:spectrum}
The solution space of $M_P$ is called the spectrum of $P$, and is denoted by $\mathcal{S}(P)$.
\end{defn}

The following lemma from \cite{ellis2014minimal}  indicates when a solution to the construction matrix will realize a graph of a particular order.

\begin{lem}\cite{ellis2014minimal}\label{lem:mintilesprop3}
Let $P=\{t_1,\ldots,t_p\}$. If $\left\langle{r_1,\ldots,r_p}\right\rangle \in \mathcal{S}(P)$, and there exists an $n \in \ZG$ such that $nr_j \in \ZG$ for all $j$, then there is a graph of order $n$ such that $G \in \mathcal{O}(P)$ using $nr_j$ tiles of type $t_j$. Let $m_P$ denote the smallest order of a graph in $\mathcal{O}(P)$.

\end{lem}

We will focus exclusively on pots with one bond-edge type,  meaning $M_P$ will be a $2 \times p$ matrix.  The following example demonstrates there may be more than one graph in the output of a pot $P$ with the same order.

\begin{ex}\label{ex:canonref}
Consider the pot $P=\{t_1=\{a^3\},,t_2=\{\hat{a}^3\}, t_3=\{\hat{a}\}\}$. The construction matrix is
\[
M_P = \begin{bmatrix}	3 & -3 & -1 & \aug & 0 \\
							1 & 1 & 1 & \aug & 1 \end{bmatrix}.
\]
To determine $\mathcal{S}(P)$, row-reduce $M_P$ to obtain
\[
\text{rref}(M_P) = \begin{bmatrix}	1 & 0 & \frac{1}{3} & \aug & \frac{1}{2} \\
							0 & 1 & \frac{2}{3} & \aug & \frac{1}{2} \end{bmatrix}.
\]
Thus we have
\[
\mathcal{S}(P) = \left\{\frac{1}{6k} \left\langle{3k-2z,3k-4z,6z}\right\rangle \mid k \in \Z^+, z \in \Q \cap \left[0,\frac{3k}{4}\right] \right\},
\]
and P realizes, for example, two nonisomorphic graphs of order 4.

\begin{figure}[H]
\begin{subfigure}{0.4\textwidth}
\begin{center}
\includegraphics[width=\textwidth]{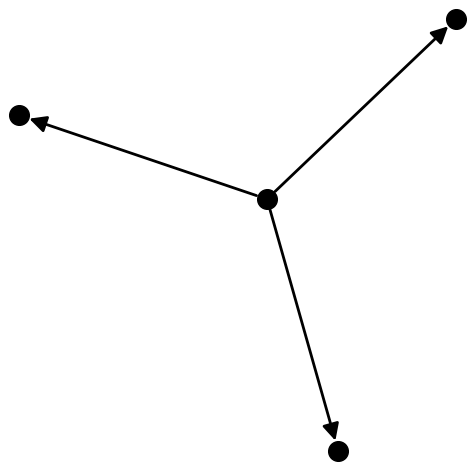}
\end{center}
\caption{Star graph of order 4 obtained from $z = \frac{3}{4}, k = 1$.}
\label{fig:example1a}
\end{subfigure}
\hfill
\begin{subfigure}{0.4\textwidth}
\begin{center}
\includegraphics[width=\textwidth]{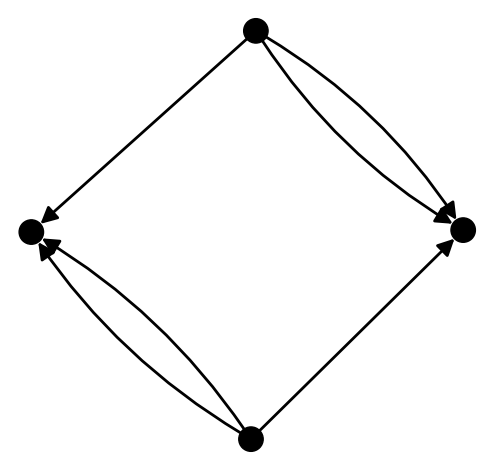}
\end{center}
\caption{Multigraph of order 4 obtained from $z = 0, k = 1$.}
\label{fig:example1b}
\end{subfigure}
\caption{A branched junction molecule and its associated tile representation.}
\label{fig:example1}
\end{figure}

\end{ex}

\section{Pots with a 1-Armed Tile} \label{sec:Results}

Since complementary cohesive-ends bond together, a tile type of the form $t = \{a^j, \hat{a}^k\}$ may form loops, leaving unmatched cohesive-ends either of the form $a^{j-k}$ or $\hat{a}^{k-j}$ depending on the magnitude of $j$ and $k$. For this paper, we only consider multigraphs without loops so we study tile types with only one cohesive-end type; that is, we restrict ourselves to the pot of tiles of the form $P_1 = \{t_1=\{a^{e_1}\},t_2 = \{\hat{a}^{e_2}\}, t_3 =\{\hat{a}\}, \ldots\}$ for some $e_1>0$ and $e_2>1$. The 1-armed tile $\{\hat{a}\}$ ensures that $P$ will always realize a complete complex. In all but Theorem \ref{thm:disc}, we assume $e_1 \geq e_2$.  Note that all of the results here can be stated identically for $P_1' = \{\{a^{e_1}\},\{\hat{a}^{e_2}\},\{\hat{a}\},\ldots\}$ where $e_1 < e_2$ by swapping the roles of $a$ and $\hat{a}$. The pot $P_1$ has corresponding construction matrix
\[
M_{P_1} = \begin{bmatrix} e_1 & -e_2 & -1 & \cdots & \aug & 0 \\ 1 & 1 & 1 & \cdots & \aug & 1 \end{bmatrix}.
\]

Unless otherwise specified, for the remainder of the paper we reserve the notation $P$ for the pot $P = \{t_1 = \{a^{e_1}\}, t_2 = \{\hat{a}^{e_2}\}, t_3 = \{\hat{a}\} \}$ because $\mathcal{O}(P) \subseteq \mathcal{O}(P_1)$. That is, any graph realized by $P$ will also be realized by the pot $P_1$. These simplifications are necessary since determining if a graph of order $n$ is realized by a pot is known to be NP-hard in general \cite{Johnson2021}.

The spectrum of $P$ is described in Lemma \ref{lem:SpectrumofP}.

\begin{lem}\label{lem:SpectrumofP}
Consider the pot $P$ with the associated construction matrix 
\begin{equation}\label{equ:conmat}
M_P = \begin{bmatrix} e_1 & -e_2 & -1 & \aug & 0 \\ 1 & 1 & 1 & \aug & 1 \end{bmatrix}.
\end{equation}
Then
\[
\mathcal{S}(P) = \left\{\frac{1}{k(e_1+e_2)}\left\langle{ke_2 - (e_2-1)z, ke_1 - (e_1+1)z, (e_1 + e_2)z}\right\rangle \mid k \in \ZG, z \in \mathbb{Q} \cap \left[0,\frac{ke_1}{e_1+1}\right]\right\}.
\]
\end{lem}

\begin{proof}

Row-reduce $M_P$ to obtain
\begin{equation}\label{lem1equ1}
\text{rref}(M_P) = \begin{bmatrix} 1 & 0 & \frac{e_2-1}{e_1+e_2} & \aug & \frac{e_2}{e_1+e_2} \\
0 & 1 & \frac{e_1+1}{e_1+e_2} & \aug & \frac{e_1}{e_1+e_2} \end{bmatrix}.
\end{equation}

From Equation \ref{lem1equ1}, we have
\begin{align}
x &= \frac{1}{e_1+e_2}[e_2-(e_2-1)z], \label{lem1equ3} \\
y &= \frac{1}{e_1+e_2}[e_1-(e_1+1)z], \label{lem1equ4} \\
z &= \frac{1}{e_1+e_2}[(e_1+e_2)z]. \label{lem1equ5}
\end{align}

Equations \ref{lem1equ3}, \ref{lem1equ4} and \ref{lem1equ5} yield the desired result.
\end{proof}

We now turn our attention to the set of graphs in the output of $P$. 
The following examples demonstrate three types of graphs that are realized by $P$ for any $e_1$ and $e_2$.

\begin{ex}\label{ex:DivAlgGraph}
   The Division Algorithm guarantees there exist unique integers $q$ and $r$ such that $e_1 = e_2q + r$ where $0 \leq r < e_2$. The pot $P$ realizes a graph of order $q + r + 1$.

\begin{proof}
    Set $z = \frac{r}{q+r+1}$ and $k = 1$. We will use the substitution $e_1 = e_2q + r$ in strategic places in this proof. Then by Lemma \ref{lem:SpectrumofP} we have the particular solution

\begin{align*}
   & \frac{1}{e_1 + e_2}\left\langle e_2 - (e_2 - 1)\left(\frac{r}{q+r+1}\right), e_1 - (e_1+1)\left(\frac{r}{q+r+1}\right), (e_1+e_2)\left(\frac{r}{q+r+1}\right) \right\rangle \\
    &= \frac{1}{e_1 + e_2}\left\langle \frac{e_2q + e_2r + e_2 - e_2r + r}{q+r+1}, \frac{e_1q + e_1r +e_1 - e_1r - r}{q+r+1}, \frac{(e_1+e_2)r}{q+r+1} \right\rangle \\
    &= \frac{1}{e_1 + e_2}\left\langle \frac{e_2q + r + e_2}{q+r+1}, \frac{e_1q +e_1 - r}{q+r+1}, \frac{(e_1+e_2)r}{q+r+1} \right\rangle \\
    &= \frac{1}{e_1 + e_2}\left\langle \frac{e_1 + e_2}{q+r+1}, \frac{e_1q + e_2q + r - r}{q+r+1}, \frac{(e_1+e_2)r}{q+r+1} \right\rangle \\
    &= \left\langle \frac{1}{q + r + 1}, \frac{q}{q+r+1}, \frac{r}{q+r+1} \right\rangle.
\end{align*}

By Lemma \ref{lem:mintilesprop3}, the graph of order $q+r+1$ has tile distribution $(1, q, r)$.

\end{proof}


\end{ex}

\begin{ex}\label{ex:StarGraph}
The pot $P$ realizes a graph of order $1 + e_1$.

\begin{proof}
    Set $z = \frac{e_1}{1 + e_1}$ and $k = 1$. Then by Lemma \ref{lem:SpectrumofP}, we have the particular solution

    \begin{align*}
   & \frac{1}{e_1 + e_2}\left\langle e_2 - (e_2 - 1)\left(\frac{e_1}{1+e_1}\right), e_1 - (e_1+1)\left(\frac{e_1}{1+e_1}\right), (e_1+e_2)\left(\frac{e_1}{1+e_1}\right) \right\rangle \\
    &= \frac{1}{e_1 + e_2}\left\langle e_2 - \frac{e_1(e_2 - 1)}{1+e_1}, e_1 - e_1, \frac{e_1(e_1+e_2)}{1+e_1} \right\rangle \\
    &= \left\langle \frac{e_2(1+e_1) - e_1(e_2-1)}{(e_1+e_2)(1+e_1)}, 0, \frac{e_1}{1 + e_1} \right\rangle \\
    &= \left\langle \frac{e_2 + e_1e_2 - e_1e_2 + e_1}{(e_1 + e_2)(1+e_1)}, 0, \frac{e_1}{1+e_1} \right\rangle \\
    &= \left\langle \frac{1}{1+e_1}, 0, \frac{e_1}{1+e_1}\right\rangle.
\end{align*}

    By Lemma \ref{lem:mintilesprop3}, the graph of order $1 + e_1$ can be realized with tile distribution $(1, 0, e_1)$.
\end{proof}
\end{ex}

\begin{ex}\label{ex:lcmGraph}
The pot $P$ realizes a graph of order $e_1 + e_2$.

\begin{proof}
    Set $z = 0$ and $k = 1$. Then by Lemma \ref{lem:SpectrumofP}, we have the particular solution $\frac{1}{e_1 + e_2}\left\langle e_2, e_1, 0\right\rangle $ and by Lemma \ref{lem:mintilesprop3}, a graph of order $e_1 + e_2$ can be realized with tile distribution $(e_2, e_1, 0)$.
\end{proof}
    
\end{ex}

\begin{ex}\label{ex:threegraphsinO(P)}
Let $P=\{\{a^9\},\{\hat{a}^6\},\{\hat{a}\}\}$. Then, according to Examples \ref{ex:DivAlgGraph}, \ref{ex:StarGraph}, and \ref{ex:lcmGraph}, $P$ will realize graphs of order 5, 10, and 15, respectively. Examples of each of these graphs is provided in Figures \ref{fig:exgraphs-1} and \ref{fig:exgraphs-2}. 

\begin{figure}[H]
\begin{subfigure}{0.4\textwidth}
\begin{center}
\includegraphics[width=0.9\textwidth]{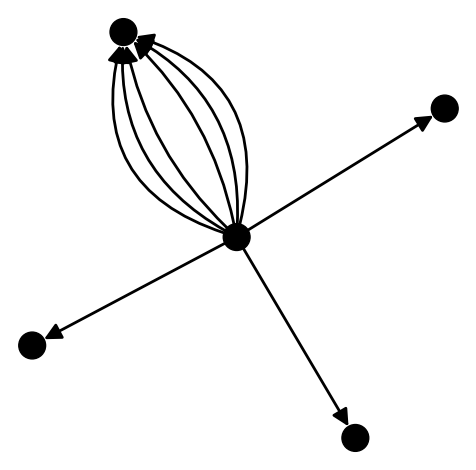}
\end{center}
\caption{Graph of order $q+r+1$}
\end{subfigure}
\hfill
\begin{subfigure}{0.4\textwidth}
\begin{center}
\includegraphics[width=0.9\textwidth]{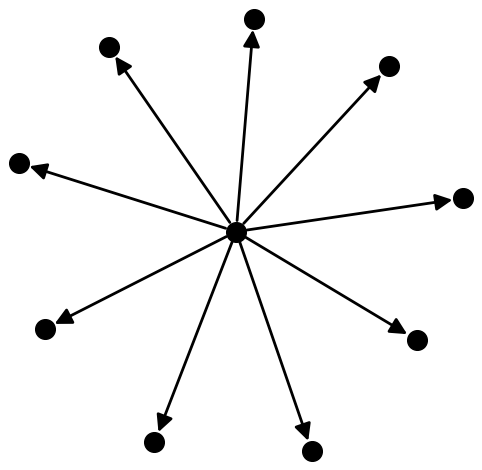}
\end{center}
\caption{Graph of order $1+e_1$}
\end{subfigure}
\caption{Two graphs for $P=\{\{a^9\},\{\hat{a}^6\},\{\hat{a}\}\}$}
\label{fig:exgraphs-1}
\end{figure}

\begin{figure}[H]
    \centering
    \includegraphics[width=0.4\textwidth]{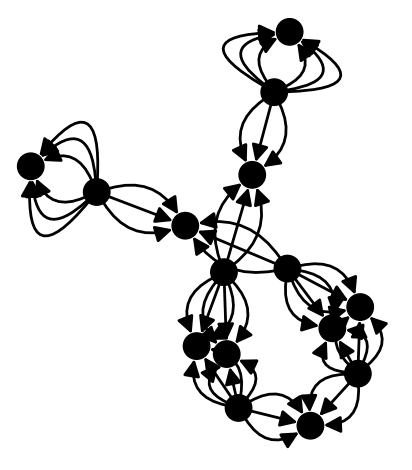}
    \caption{Graph of order $e_1+e_2$ for $P = \{\{a^9\},\{\hat{a}^6\},\{\hat{a}\}\}$}
    \label{fig:exgraphs-2}
\end{figure}
\end{ex}


\subsection{Connections Between $e_1$, $e_2$, and $\mathcal{O}(P)$}

In this section, we demonstrate how certain relationships between $e_1$ and $e_2$ will determine the orders of the graphs in $\mathcal{O}(P)$. In particular,
We show that if $G \in \mathcal{O}(P)$, then the order of $G$ is dependent upon $\text{gcd}(e_1+1,-e_2+1)$. The most straightforward case is presented in our first theorem which states the conditions in which $P$ will realize graphs of orders that are multiples of  $\text{gcd}(e_1+1,-e_2+1)$.

\begin{thm}\label{mainthm}
For the pot $P$, if $\text{gcd}(e_1+1,-e_2+1)=d \ne 1$, then $P$ realizes a graph of order $n$ if and only if $n=kd$ where $k \in \ZG$ and $n \geq m_P$.
\end{thm}

\begin{proof}


 From Equation \ref{equ:conmat},  we have the system of equations \begin{equation}\label{equ:system}\left\{
\begin{array}{l}
e_1x-e_2y-z=0, \\
x+y+z=n,
\end{array}
\right.\end{equation}
where $x,y,z \in \ZG$.


Adding these equations, we obtain
\begin{equation}\label{dioequ}
(e_1+1)x+(-e_2+1)y=n.
\end{equation}
This is a linear Diophantine equation in two variables and since $\text{gcd}(e_1+1,-e_2+1)=d \ne 1$, a solution to this equation exists if and only if $n=kd$, which establishes the desired result.
\end{proof}


\begin{cor}\label{cor:e1_odd}
    Let $P$ be a pot where $e_1$ is odd and $e_1 = e_2$. Then $P$ realizes a graph for all orders $n$ where $n \in 2\ZG$.
\end{cor}


\begin{proof}

From Theorem \ref{mainthm}, it is sufficient to notice that since $e_1$ is odd and $e_1 = e_2$, $\text{gcd}(e_1 + 1, -e_1 + 1) = 2$. Hence all graphs realized by $P$ must have order $2k$ for $k \in \ZG$.


\end{proof}

\begin{ex}
Let $P = \{\{a^3\},\{\hat{a}^3\},\{\hat{a}\}\}$. Then by Corollary \ref{cor:e1_odd}, $P$ realizes a graph of every even order. Two distinct connected graphs are provided below: the graph of minimal order 2, and a graph of order 6.

\begin{figure}[H]
\begin{subfigure}{0.4\textwidth}
\begin{center}
\includegraphics[width=0.15\textwidth]{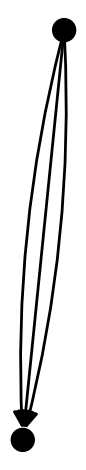}
\end{center}
\caption{Graph of order 2}
\end{subfigure}
\begin{subfigure}{0.5\textwidth}
\begin{center}
\includegraphics[width=0.75\textwidth]{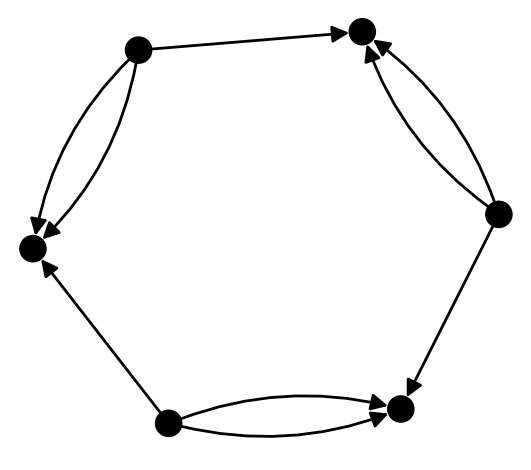}
\end{center}
\caption{Graph of order 6}
\end{subfigure}
\caption{Two Graphs for $P=\{\{a^3\},\{\hat{a}^3\},\{\hat{a}\}\}$}
\end{figure}
\end{ex}




The case when $e_1 = e_2$ is even is not as immediate since $gcd(e_1 + 1, -e_1 + 1) = 1$. We provide a motivating example in which $gcd(e_1 + 1, -e_2 + 1) = 1$.

\begin{ex}\label{canonex}
Consider the pot $P=\{\{a^6\},\{\hat{a}^4\},\{\hat{a}\}\}$. The associated construction matrix is
\[
M_P=\begin{bmatrix} 6 & -4 & -1 & \aug & 0 \\ 1 & 1 & 1 & \aug & 1 \end{bmatrix}
\]
and $\mathcal{S}(P) = \{\frac{1}{10k} \langle 4k - 3z, 6k - 7z, 10z\rangle \mid k \in \ZG, z \in \mathbb{Q} \cap [0, \frac{6k}{7}] \}$. When $k = 1$ and $z = \frac{1}{2}$ we obtain a graph of order 4 with tile distribution $(1,1,2)$, which is a graph of minimal order in $\mathcal{O}(P)$ (see Figure \ref{fig:614_min}). Notice that from Equation \ref{dioequ}, there is an associated linear Diophantine equation $7x - 3y = n$. Software can be used to verify that the pot $P$ does not realize a graph of every order; i.e. the results of Theorem \ref{mainthm} do not generalize when $gcd(e_1 + 1, -e_2 + 1) = 1$.




\begin{figure}[H]
\begin{center}
\includegraphics[width=0.3\textwidth]{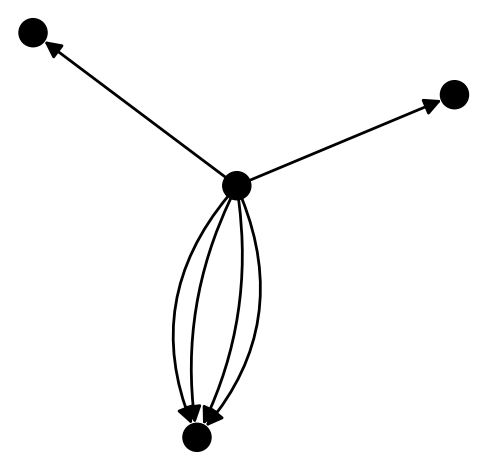}
\end{center}
\caption{Graph of order 4 for $\{\{a^6\},\{\hat{a}^4\},\{\hat{a}\}\}$}\label{fig:614_min}
\end{figure}


\end{ex}


There are some important observations from Example \ref{canonex}. It appears $P$ realizes a graph for all orders $n$, except when $n = 1, 2, 3, 6$. Notice the smallest order graph realized by $P$ is order 4, but $P$ does not realize a graph of order 5. Theorem \ref{thm:lwrbnd} generalizes the idea that there is some lower bound after which $P$ will realize a graph of any order and this lower bound may not be the order of the smallest graph realized by $P$.

\begin{defn}\label{zeta}
Let  $S_P = \{n \mid n+k \text{ is the order of some }  G \in \mathcal{O}(P) \text{ for every } k \in \N\}$. The \textit{lower density bound} of $P$ is $\zeta = \text{min}(S_P)$.
\end{defn}


That is, the lower density bound $\zeta$ is the smallest order for which $P$ realizes a graph of every order larger than $\zeta$.

In general, it is difficult to predict $\zeta$ for $P$. However, Theorem \ref{thm:lwrbnd} provides a lower bound that is close to $\zeta$. 

\begin{thm}\label{thm:lwrbnd}
Consider the pot $P$ where $\text{gcd}(e_1+1,-e_2+1)=1$. Then $P$ realizes a graph for every order $n$ with $n \geq \text{max}\left\{\frac{(e_1+1)(e_1+e_2)}{e_1},\frac{(e_2-1)(e_1+e_2)}{e_2}\right\}$.
\end{thm}

\begin{proof}

From Equation \ref{dioequ}, $P$ realizes a graph $G$ of order $n$ if and only if $(e_1+1)x + (-e_2+1)y = n$ for some $x$ and $y$. Solving for $y$ and defining $y=f(x)$ gives the function
\[
f(x)=\frac{n-(e_1+1)x}{-e_2+1}.
\]

Since we only consider nonnegative solutions $(x,f(x))$ to Equation \ref{dioequ} (i.e. $x \geq 0$ and $y=f(x) \geq 0$), then we will find upper bounds for $x$ and $f(x)$ to guarantee that a solution exists.
The key observation is to notice $z=n-(x+f(x))$ from Equation \ref{equ:system}. Thus, substituting for $f(x)$ we have 
\[
z = \frac{-e_2n+(e_1+e_2)x}{-e_2+1}.
\]

To find the upper bound on $x$ and $f(x)$, we determine the value for which $z=0$.
When $z \geq 0$, we have $x \leq \frac{e_2n}{e_1+e_2}$. This provides the bounds

\[\left\{
\begin{array}{l}
0 \leq x \leq \frac{e_2n}{e_1+e_2}, \\
0 \leq f(x) \leq f\left(\frac{e_2n}{e_1+e_2}\right) = \frac{e_1n}{e_1+e_2}.
\end{array}
\right.\]

The slope of $f(x)$ is $\frac{e_1+1}{e_2-1}$. Let $(x_1,f(x_1)) = \text{min} \{(x, f(x)) \in \Z^2 \mid x \geq e_2 - 1 \text{ and } f(x) \geq e_1 + 1\} $. The slope of $f(x)$ guarantees
\[\left\{\begin{array}{l}
0 \leq x_1-(e_2-1) < e_2-1, \\
0 \leq f(x_1)-(e_1+1) < e_1+1,
\end{array}\right.\]
with $(x_1-(e_2-1),f(x_1)-(e_1+1)) \in \Z^2$. Thus an integer point is guaranteed if the inequalities
\[\left\{\begin{array}{l}
e_2-1 \leq \frac{e_2n}{e_1+e_2}, \\
e_1+1 \leq \frac{e_1n}{e_1+e_2},
\end{array}\right.\]
are both satisfied (see Figure \ref{fig:boundsrect}). Thus, by solving both inequalities for $n$, we conclude that $P$ realizes a graph for every $n$ with $n \geq \text{max}\left\{\frac{(e_1+1)(e_1+e_2)}{e_1},\frac{(e_2-1)(e_1+e_2)}{e_2}\right\}$.


\end{proof}

\begin{figure}[H]
\begin{center}
\includegraphics[scale=0.5]{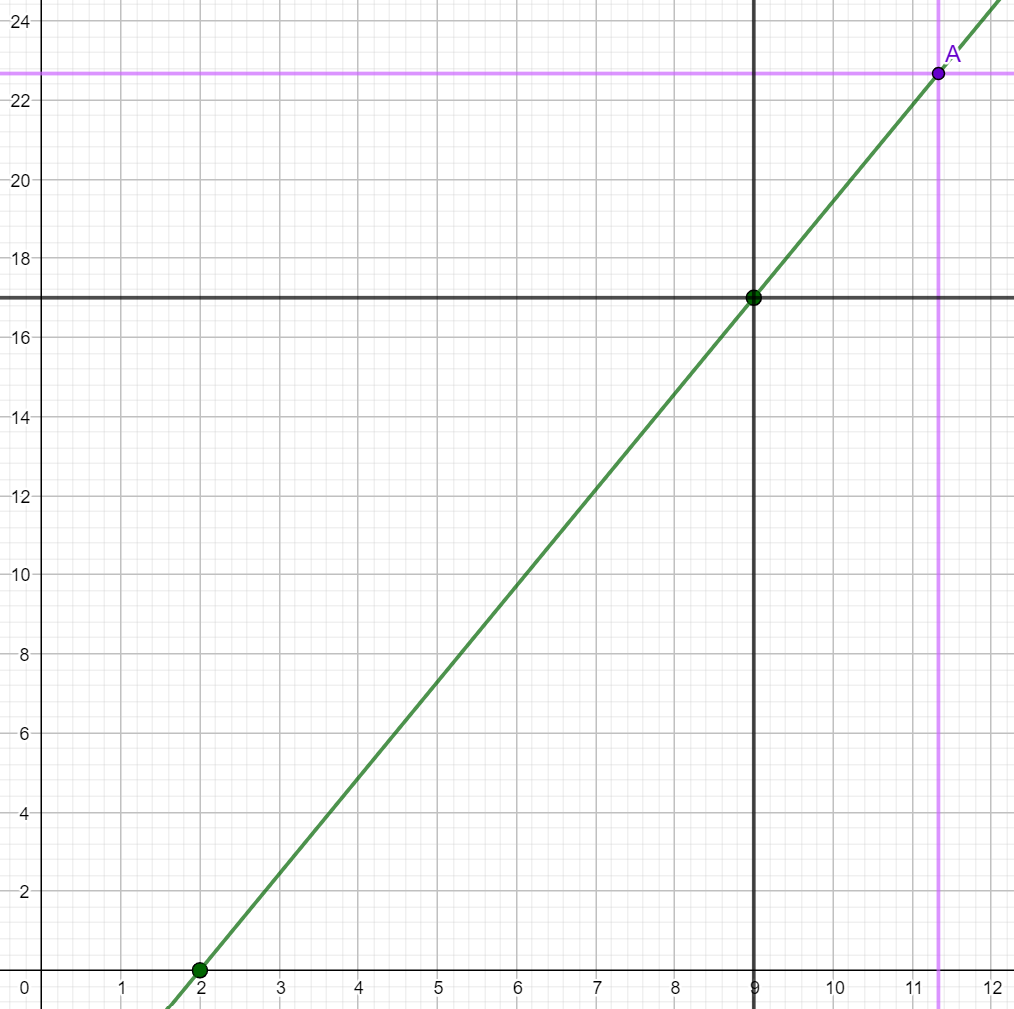}
\end{center}
\caption{The purple rightmost vertical line and upper horizontal line represent the bounds on $x$ and $f(x)$, respectively, while the black lines represent $y=e_1+1$ and $x=e_2-1$. The green line with positive slope is the function $f(x)$. }
\label{fig:boundsrect}
\end{figure}

\begin{rem}
We denote the lower bound derived in Theorem \ref{thm:lwrbnd} by $\eta$. That is, $$\eta=\text{max}\left\{\frac{(e_1+1)(e_1+e_2)}{e_1},\frac{(e_2-1)(e_1+e_2)}{e_2}\right\}.$$
\end{rem}

Despite the fact that $\zeta \leq \eta$, there are only finitely many orders to check for a pot $P$ to determine the value of $\zeta$. That is, one need only check all orders for $m_P \leq n \leq \eta$, which can be done using software.

\subsection{Connected and Disconnected Graphs}

Knowing orders of the graphs that can be realized by $P$ allows us to address the next research question related to the types of graphs realized by $P$. The following theorem demonstrates when an arbitrary pot realizes a disconnected graph. Notice that the theorem below is applicable to any pot of tiles with any number of bond-edge types.

\begin{thm}\label{thm:disc}
    The pot $P_j = \{t_1, t_2, \ldots, t_j\}$ realizes a disconnected graph $G$ of order $n$ if and only if for at least one tile distribution $(R_{1},\ldots,R_{j})$ associated to $G$, 
\[
(R_{1},\ldots,R_{j}) = \sum_{i=1}^m (R_{1i},\ldots,R_{ji}),
\]
where each $j$-tuple $(R_{1i},\ldots,R_{ji})$ is a tile distribution of $P_j$ that realizes a graph of order less than $n$.
\end{thm}

\begin{proof}
    Suppose $P_j$ realizes a disconnected graph $G$ of order $n$ and let $(R_{1},\ldots,R_{j})$ be the tile distribution which realizes $G$. Then the graph $G$ is a union of disjoint subgraphs; that is,
\begin{equation}\label{eqn:tileuni}
G = \bigcup_{i=1}^m H_i
\end{equation}
where $V(H_i) \cap V(H_j) = \varnothing$ for $i \ne j$. Since each $H_i$ is a graph, and hence a complete complex, there must be some tile distribution, namely $(R_{1i},\ldots,R_{ji})$, that realizes $H_i$. Further, since each $H_i \subset G$, it follows that $R_{ki} \leq R_{k}$ for all $1 \leq k \leq j$. Thus we translate Equation \ref{eqn:tileuni} into the language of tile distributions to arrive at 
\[
(R_{1},\ldots,R_{j}) = \sum_{i=1}^m(R_{1i},\ldots,R_{ji}).
\]

Conversely, suppose $P_j$ realizes graphs $H_i$ with corresponding tile distributions $(R_{1i},\ldots,R_{ji})$ such that
\[
(R_{1},\ldots,R_{j}) = \sum_{i=1}^m (R_{1i},\ldots,R_{ji}) \text{ and } R_1 + R_2 + \cdots + R_j = n,
\]
then it follows immediately that $P_j$ realizes a disconnected graph of order $n$ with tile distribution $(R_{1},\ldots,R_{j})$.
\end{proof}

Theorem \ref{thm:disc} provides the conditions under which the pot $P = \{\{a^{e_1}\}, \{\hat{a}^{e_2}\}, \{\hat{a}\}\}$ will realize disconnected graphs as shown by the following two corollaries. 

\begin{cor}
Let $\text{gcd}(e_1+1,-e_2+1)=d \ne 1$. The pot $P$ realizes a disconnected graph of order $n$ if and only if
\[
n = n_1+n_2+\cdots+n_{\ell}
\]
where $n_i = k_id$ for some $k_i \in \ZG$ and $n_i \geq m_P$.
\end{cor}

\begin{proof}

The proof is immediate from Theorem \ref{mainthm} and Theorem \ref{thm:disc}.
\end{proof}



\begin{cor}\label{cor:DisconnectedGCD1}
    For the pot $P$, if $\text{gcd}(e_1+1,-e_2+1)=1$ and $n = n_1 + n_2 + \cdots + n_{\ell}$ where   each $n_i \geq \zeta$, then $P$ realizes a disconnected graph of order $n$.
\end{cor}

\begin{proof}
The proof is immediate from Definition \ref{zeta} and Theorem \ref{thm:disc}.
\end{proof}

\begin{ex}



Consider the pot $P=\{\{a^7\},\{\hat{a}^4\},\{\hat{a}\}\}$. From computational software, it appears that $\zeta = 7$; that is, $P$ realizes a graph of every order greater than or equal to order 7. Thus, $P$ will realize a disconnected graph of order 15 in which one component is a subgraph of order 7 and the other component is a subgraph of order 8. However,  $P$ will also realize a disconnected graph of order 12 (See Figure \ref{fig:714s12}). This shows that the converse of Corollary \ref{cor:DisconnectedGCD1} is not necessarily true.

\begin{figure}[H]
\begin{minipage}{0.4\textwidth}
\begin{center}
\includegraphics[width=0.75\textwidth]{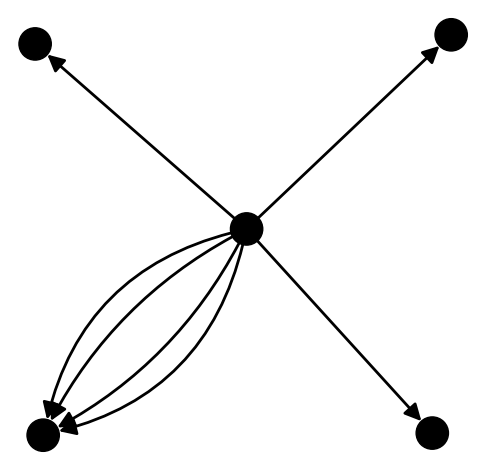}
\end{center}
\end{minipage}
\hfill
\begin{minipage}{0.4\textwidth}
\begin{center}
\includegraphics[width=0.45\textwidth]{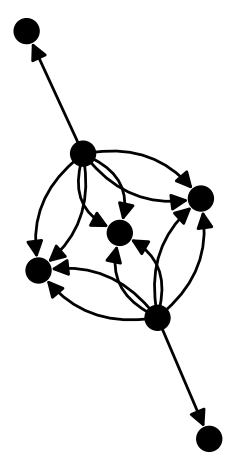}
\end{center}
\end{minipage}
\caption{Disconnected graph of order 12 for $\{\{a^7\},\{\hat{a}^4\},\{\hat{a}\}\}$}
\label{fig:714s12}
\end{figure}
\end{ex}

Given a tile distribution $(R_1, R_2, R_3)$ for the pot $P$, the following theorem establishes a relationship between $R_1$ and $R_2$ that indicates when any graph realized by $P$ is necessarily disconnected. 

\begin{thm}\label{thm:pathdisc}
Let $G \in \mathcal{O}(P)$ and let $(R_1, R_2, R_3)$ be a tile distribution that constructs $G$. If $1+R_2(e_2-1) < R_1$, then $G$ is disconnected.
\end{thm}

\begin{proof}
 Suppose $G \in \mathcal{O}(P)$ where $G = H_1 \cup H_2$ and $H_1 \cap H_2 = \emptyset$. 
 We will show that despite maximizing the number of vertices of tile-type $t_1$ in $H_1$, the subgraph $H_2$ will be nonempty (i.e. $G$ is disconnected). 

If $H_2$ is empty, then $H_1$ must contain $R_2$ vertices of tile type $t_2$. To maximize the number of vertices of tile-type $t_1$, $H_1$ must be acylic; i.e. $H_1$ is a tree. Consider the subgraph of $H_1$ whose vertex set is only vertices of tile-types $t_1$ and $t_2$; call this subgraph $H_1'$. Since there are $R_2$ vertices of tile-type $t_2$,  there must be $e_2R_2$ edges  and $e_2R_2 + 1$ vertices in $H_1'$. Thus, there can be at most $e_2R_2 + 1 - R_2$ vertices of tile-type $t_1$ in $H_1'$ and hence, in $H_1$. 

If $1 + R_2(e_2-1) < R_1$, then $V(H_2)$ must contain vertices labeled with tile-type  $t_1$. Therefore, $G$ will be disconnected.
\end{proof}


The question of when connected graphs are realized by $P$ remains open at the time of writing. The algorithms in Section \ref{sec:Algorithms} provide a partial answer to this question.

\section{Connected Graph Algorithms} \label{sec:Algorithms}
Section \ref{sec:Results} establishes the conditions under which a pot of the form $P = \{\{a^{e_1}\}, \{\hat{a}^{e_2}\}, \{\hat{a}\}\}$  realizes a connected or disconnected graph. 
In this section, we provide two algorithms which will construct a connected graph from $P$. 
Theorem \ref{thm:pathdisc} suggests a connected graph may exist if $1 + R_2(e_2-1) \geq R_1$ and the desired order for a graph satisfies Theorem \ref{mainthm} or \ref{thm:lwrbnd}; the algorithms that follow rely on this inequality. 

\subsection{Path Algorithm}
\label{Alg:Path}

If $1 + R_2(e_2-1) \geq R_1$, then the following algorithm will output a connected graph. Note that the algorithm has been written based upon the assumption that $R_1 \geq R_2$. If $R_1 < R_2$, then the roles of $R_1$ and $R_2$, and $t_1$ and $t_2$ can be swapped in steps 1, 2, and 3 in order to produce a connected graph.

\begin{center}
	{\parbox{5.5in}{\raggedright
		\textbf{Input:} A pot $P = \{t_1 = \{a^{e_1}\}, t_2 = \{\hat{a}^{e_2}\}, t_3 =\{\hat{a}\}\}$ with corresponding tile distribution $(R_1, R_2, R_3)$  \\
		\textbf{Output:} A labeled connected graph of order $R_1 + R_2 + R_3$}}
\end{center}

\begin{enumerate}
    \item Form a path graph on $2R_2-1$ vertices where the vertices are labeled as in Figure \ref{fig:pathgraph}. That is, let \[\lambda(v_k)=
    \left\{\begin{array}{l}
    t_2 \text{ if } k \text{ odd} \\
    t_1 \text{ if } k \text{ even},
    \end{array}\right.\]
    where $k \in \{1,\ldots,2R_2-1\}.$
    
    
    \begin{figure}[H]
    \begin{center}
    \includegraphics[width=0.7\textwidth]{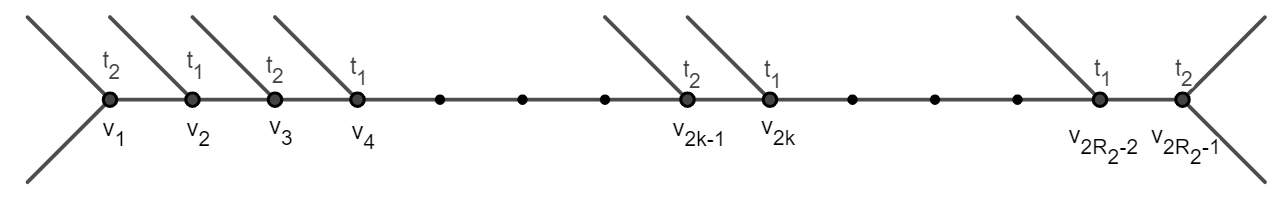}
    \end{center}
    \caption{Graph after step 1 of Algorithm \ref{Alg:Path} including unmatched half-edges on $t_1$ and $t_2$}
    \label{fig:pathgraph}
    \end{figure}
    
    \item  If $R_1 - (R_2 - 1) < e_2 - 1$, then attach one half-edge from each of the remaining $R_1 - (R_2 - 1)$ copies of $t_1$ to vertex $v_1$. Go to step \ref{Alg:AfterLoop}. Else, attach one half-edge from each of $e_2-1$ copies of $t_1$ to vertex $v_1$. Set $\texttt{counter} = R_1 - (R_2 - 1) - (e_2-1)$.
   
    \item  If $\texttt{counter} < e_2 - 2$, then attach one half-edge from each of the remaining $\texttt{counter}$ copies of $t_1$ to vertex $v_3$. Else, attach one half-edge from each of $e_2 - 2$ copies of $t_1$ to vertex $v_3$. Update $\texttt{counter} = \texttt{counter}  - (e_2-2)$.  Continue in this way sequentially for vertices $v_5$, $v_7, \ldots$ until $\texttt{counter} = 0$. Note if $\texttt{counter} > 0$  at the end of the sequence (i.e. at vertex $v_{2R_2 -1}$) it may be necessary to attach one half-edge from $e_2 - 1$ copies of $t_1$ to vertex $v_{2R_2 -1}$.
   
    \item \label{Alg:AfterLoop} Attach any unpaired half-edges from the $t_2$ tile types to any unpaired half-edges from the $t_1$ tile types. 
 
    \item 
    Attach each $t_3$ to an unpaired half-edge from $t_1$. 
\end{enumerate}

\begin{proof}
Let $G$ be a graph constructed by Algorithm \ref{Alg:Path}. We note that by construction, $R_2$ vertices are labeled $t_2$ in Step 1. At the end of Step 3, $R_1$ vertices are labeled $t_1$ and at the end of Step 5, $R_3$ vertices are labeled $t_3$. Thus, the graph constructed is realized by $P$.

We next prove that $G$ is a connected graph by showing there are no unmatched half-edges by the end of the algorithm.

In the multiset $P_\lambda(G)$, 
there are $e_1R_1$ total half-edges labeled $a$ associated to tile type $t_1$. In Step 1 of the algorithm, there are $2(R_2 - 1)$ half-edges labeled $a$ that are joined to half-edges labeled $\hat{a}$, and in Steps 2 and 3, there are an additional $R_1 - (R_2 - 1)$ half-edges labeled $a$ joined to half-edges labeled $\hat{a}$. At the end of Step 3, there are $e_1R_1 - 2(R_2 - 1) - (R_1 - (R_2 - 1)) = (e_1-1)R_1 - R_2 + 1$ half-edges labeled $a$ that are unmatched.

In $P_\lambda(G)$, there are $e_2R_2$ total half-edges labeled $\hat{a}$ associated to tile type $t_2$. In Step 1, $2 + 2(R_2 - 2)$ half-edges labeled $\hat{a}$ are joined to half-edges labeled $a$. In Steps 2 and 3, $R_1 - (R_2 - 1)$ half-edges labeled $\hat{a}$ are joined to half-edges labeled $a$. Thus, at the end of Step 3, there are $e_2R_2 - (2 + 2(R_2 - 2)) - (R_1 - (R_2 - 1) = (e_2-1)R_2 - R_1 + 1$ half-edges labeled $\hat{a}$ that are unmatched.

Note that 
    \[
        (e_1-1)R_1 - R_2 + 1 - \left((e_2-1)R_2 - R_1 + 1\right) = e_1R_1 - e_2R_2. 
    \]
Since $e_1 > e_2$ and $R_1 \geq R_2$, then $e_1R_1 - e_2R_2 > 0$. That is, there must be more unmatched half-edges labeled $a$ than $\hat{a}$ at the end of Step 3. This ensures that all unmatched half-edges labeled $\hat{a}$ will be joined to half-edges labeled $a$ in Step 4. 
    
There will be exactly $e_1R_1 - e_2R_2$ unmatched half-edges labeled $a$ at the end of Step 4. Since $e_1R_1 - e_2R_2 = R_3$ by Equation \ref{equ:system}, then Step 5 guarantees all remaining half-edges can be matched to a vertex labeled $t_3$. Thus, there are no unmatched half-edges by the end of the algorithm.

\end{proof}

\begin{thm}
If $1+R_2(e_2-1) \geq R_1$, then there exists a connected graph $G \in \mathcal{O}(P)$.
\end{thm}

The proof of this theorem is immediate from Algorithm \ref{Alg:Path}.

\begin{ex}
    Consider $P = \{\{a^4\}, \{\hat{a}^3\}, \{\hat{a}\}\}$. $P$ realizes a graph of order 29 with a tile distribution of $(7,3,19)$. The output of Algorithm \ref{Alg:Path} is shown in Figure \ref{fig:pathex}.

    \begin{figure}[H]
        \centering
        \includegraphics[width=0.5\textwidth]{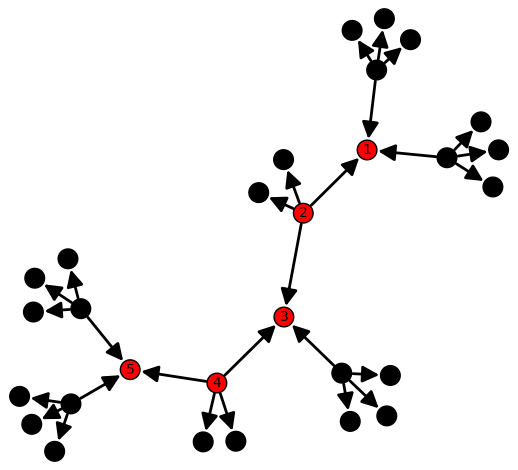}
        \caption{A graph of order 29 constructed using Algorithm \ref{Alg:Path}}
        \label{fig:pathex}
    \end{figure}
\end{ex}


\subsection{Cycle Algorithm}\label{Alg:cycle}

Ring structures naturally occur in many biological systems \cite{chen2022rings,ertl2021rings}. 
For this reason, we have created an algorithm to construct a connected graph from a cycle graph. The limitation that occurs is that this algorithm only works when $R_1 \leq R_2$.


\begin{center}
	{\parbox{5.5in}{\raggedright
		\textbf{Input:} A pot of the form $P = \{t_1 = \{a^{e_1}\}, t_2 = \{\hat{a}^{e_2}\}, t_3 = \{\hat{a}\}\}$ with corresponding tile distribution $R_1,R_2,R_3$. \\
		\textbf{Output:} A labeled connected graph of order $R_1 + R_2 + R_3$.}}
\end{center}

\begin{enumerate}
    \item Form a cycle graph on $2R_1$ vertices. Alternate the labels on the vertices for $t_1$ and $t_2$. That is, choose a vertex to be $v_1$, then set $\lambda(v_{2k-1}) = t_1$ and $\lambda({v_{2k}}) = t_2$ for $k \in \{1,\ldots,R_1\}$. See Figure \ref{fig:cydia}.
    
    \begin{figure}[H]
    \begin{center}
    \includegraphics[width=0.5\textwidth]{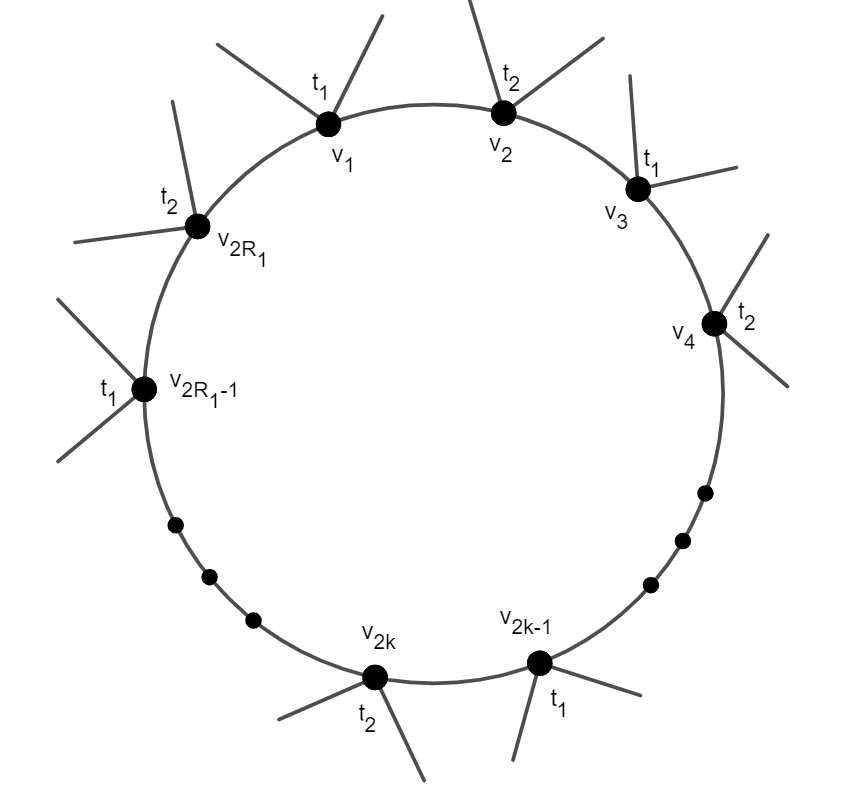}
    \end{center}
    \caption{Graph after step 1 of Algorithm \ref{Alg:cycle} including unmatched half-edges on $t_1$ and $t_2$}
    \label{fig:cydia}
    \end{figure}
    
    \item Attach $\lfloor \frac{e_2-2}{2} \rfloor$ half-edges from $v_{2k}$ to $v_{2k-1}$ and attach $\lceil \frac{e_2-2}{2} \rceil$ half-edges from $v_{2k}$ to $v_{2k+1}$ for each $k \in \{1,\ldots,R_1\}$. Due to the cyclic subgraph from step 1, we note $v_1 = v_{2k+1}$ for this process.
    \item \label{Alg:matcht2} 
    If $R_2 - R_1 < e_1 - e_2$, then attach the remaining $R_2 - R_1$ copies of $t_2$ to the half-edges of $t_1$ on $v_1$ using exactly one half-edge for each copy of $t_2$. Go to step \ref{Alg:AfterLoop1}. Else, attach $e_1 - e_2$ copies of $t_2$ to the half-edges of $t_1$ on $v_1$ using exactly one half-edge of each copy of $t_2$. Update $\texttt{counter} = R_2 - R_1 - (e_1 - e_2)$ 
    \item \label{Alg:ExtraT3}If $\texttt{counter} < e_1 - e_2$, then attach \texttt{counter} copies of $t_2$ to the half-edges of $t_1$ on $v_3$ using exactly one half-edge of each copy of $t_2$. Else, attach $e_1 - e_2$ copies of $t_2$ to the half-edges of $t_1$ on $v_3$ using exactly one half-edge for each copy of $t_2$. Update $\texttt{counter} = \texttt{counter} - (e_1 - e_2)$. Continue in this way sequentially for vertices $v_5, v_7, \ldots$ until $\texttt{counter} = 0$. 
    \item \label{Alg:AfterLoop1} If $R_1 \ne R_2$, attach $(R_2-R_1)(e_2-1)$ unmatched half-edges from the vertices labeled with $t_1$ to the unmatched half-edges of the vertices labeled with $t_2$. Else, move to step \ref{Alg:Not3s}.
    \item \label{Alg:Not3s} Attach $R_1(e_1-e_2) - e_2(R_2 - R_1)$ half-edges labeled $a$ from vertices labeled with $t_1$ to $R_3$ vertices labeled with $t_3$.
\end{enumerate}
\begin{proof}

Form the graph $G$ with the labeling prescribed by Step 1 of Algorithm \ref{Alg:cycle}. Our aim is to show that all the remaining tiles can be attached to the vertices in $G$. 

After steps 1 and 2, there are $ e_1 - \left(\lfloor \frac{e_2-2}{2} \rfloor + \lceil \frac{e_2-2}{2} \rceil \right) - 2 = e_1 - (e_2-2) - 2 = e_1 - e_2$ unpaired half-edges on \textit{each} vertex labeled $t_1$. Thus, there is a total of $R_1(e_1-e_2)$ unmatched $a$'s. Additionally, all tiles of type $t_1$  and $R_1$ tiles of type $t_2$ have been used to label the vertices of $G$, and all of the arms of the tiles of type $t_2$ in $G$ have been matched.

We must ensure there are sufficiently many unmatched edges labeled $a$ for the remaining tiles of type $t_2$, which is equivalent to showing $R_2 - R_1 < R_1 (e_1-e_2)$. This can be shown using the substitution $e_1R_1 = e_2R_2 + R_3$ from Equation \ref{equ:system} as follows:

    \begin{align*}
        R_1 (e_1-e_2) &= e_1R_1 - e_2R_1 \\
        &= R_3 + e_2R_2 - e_2R_1 \\
        &= R_3 + e_2(R_2 - R_1) \\
        &\geq e_2(R_2 - R_1) \\
        &> R_2 - R_1.
    \end{align*}

After step 5, the tiles of type $t_2$ that were added in steps \ref{Alg:matcht2} and \ref{Alg:ExtraT3} have $(R_2-R_1)(e_2-1)$ unmatched half-edges. 
To ensure each half-edge can be matched to a free half-edge on a tile of type $t_1$, we must show $(R_2-R_1)(e_2-1) \leq R_1(e_1-e_2)-(R_2-R_1)$. Using the substitution $e_1R_1 - e_2R_2 = R_3$, we have:

    \begin{align*}
        R_1(e_1-e_2)-(R_2-R_1) &= e_1R_1 - e_2R_1 - R_2 + R_1 \\
        &= e_1R_1 - e_2R_1 - e_2R_2 + e_2R_2 - R_2 + R_1 \\
        &= e_1R_1 - e_2R_2 + (e_2R_2 - R_2 - e_2R_1 + R_1) \\
        &= R_3 + (R_2-R_1)(e_2-1) \\
        &\geq (R_2-R_1)(e_2-1).
    \end{align*}
    
With all of the arms from tiles of type $t_2$ matched, we finally need to check that there are exactly the number of unmatched half-edges on tiles of type $t_1$ as there are tiles of type $t_3$. That is, we need to show $R_3 = R_1(e_1-e_2)-e_2(R_2-R_1)$:
    
    \begin{align*}
        R_1(e_1-e_2)-e_2(R_2-R_1) &= e_1R_1 - e_2R_1 - e_2R_2 + e_2R_1 \\
        &= e_1R_1 - e_2R_2 \\
        &= R_3.
    \end{align*}

Thus, $G$ is a connected graph and $G \in \mathcal{O}(P)$ using tile distribution $(R_1, R_2, R_3)$.
\end{proof} 

\begin{ex}
    Consider $P = \{\{a^6\}, \{\hat{a}^4\}, \{\hat{a}\}\}$. Then $P$ realizes a graph of order 19 with a tile distribution of $(7,10,2)$. The output of Algorithm \ref{Alg:cycle} is shown in Figure \ref{fig:cycleex}.

    \begin{figure}[H]
        \centering
        \includegraphics[width=0.4\textwidth]{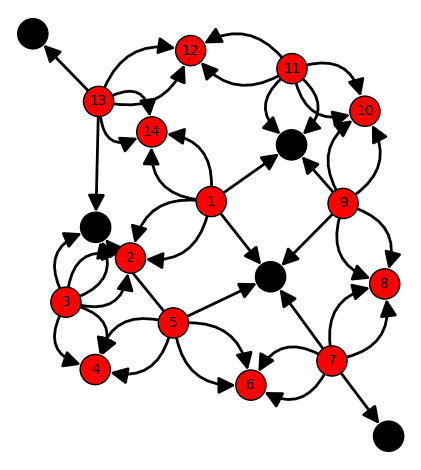}
        \caption{A graph of order 19 constructed using Algorithm \ref{Alg:cycle}}
        \label{fig:cycleex}
    \end{figure}
\end{ex}




\section{Conclusion} \label{sec:Conclusion}

We have shown that, given a pot of tiles with one bond-edge type and a 1-armed tile, we can determine the orders of the complete complexes that can be realized by the pot. To a lesser extent, we can also characterize whether these complete complexes will be disconnected or connected complexes. At the time of writing, the entire case involving a pot with one bond-edge type (i.e. a $2 \times p$ construction matrix) is close to being completely understood. Three primary questions remain to be explored:

\begin{enumerate}

\item What is a formula for $\zeta$ in terms of $e_1$ and $e_2$?

\item Does there exist a pot where $\zeta = \eta$?

\item Do these results extend to pots of the form $P=\{\{a^{e_1}\},\{\hat{a}^{e_2}\},\{\hat{a}^{e_3}\}\}$ where $1 < e_3 < e_2$?

\end{enumerate}

Although the first question remains open, our results provide a lower bound which is ``close" to $\zeta$. This means that, for any pot $P$ satisfying the relatively prime condition, there are only finitely many orders to check between the order of the minimal graph of $P$ and the corresponding $\eta$.

With the third question, we have some indication that the results in this  paper extend to pots that do not possess a 1-armed tile, but more research is needed in this area. Considering the conditions that occur when $\text{gcd}(e_1+1,-e_2+1)=d \ne 1$, it would be reasonable to start in this setting rather than the relatively prime setting.

The difficulty of determining which graphs $G$ are in $\mathcal{O}(P)$ increases dramatically when moving from pots with one bond-edge type to pots with two bond-edge types. Preliminary research suggests the results here do not necessarily generalize to the two bond-edge type case. 

\newpage

\bibliographystyle{abbrvurl}
\bibliography{bib}

\end{document}